\documentclass[12pt]{article}

\usepackage{tikz}
\usepackage{subfigure}
\usepackage[english]{babel}

\usepackage[center]{caption2}
\usepackage{amsfonts,amssymb,amsmath,latexsym,amsthm}
\usepackage{multirow}
\usepackage{threeparttable}
\usepackage{cases}

\newtheorem{thm}{Theorem}

\newtheorem{lem}[thm]{Lemma}

\newtheorem{conj}[thm]{Conjecture}

\newtheorem{claim}{Claim}

\newtheorem{problem}{Problem}

\begin{document}

\title{The size of $3$-uniform hypergraphs with given matching number and codegree\thanks{The work was supported by NNSF of China (No. 11671376) and  NSF of Anhui Province (No. 1708085MA18).}}
\author{Xinmin Hou$^a$, \quad Lei Yu$^b$, \quad Jun Gao$^c$, \quad Boyuan Liu$^d$\\
\small $^{a,b,c,d}$ Key Laboratory of Wu Wen-Tsun Mathematics\\
\small School of Mathematical Sciences\\
\small University of Science and Technology of China\\
\small Hefei, Anhui 230026, China.\\
}
\date{}
\maketitle

\begin{abstract}
Determine the size of $r$-graphs with given graph parameters is an interesting problem. Chv\'atal and Hanson (JCTB, 1976) gave a tight upper bound of the size of 2-graphs  with restricted maximum degree and matching number; Khare (DM, 2014) studied the same problem for linear $3$-graphs with restricted matching number and maximum degree.  In this paper, we give a tight upper bound of the size of $3$-graphs with bounded codegree and matching number.
\end{abstract}

\section{Introduction}
Let $r\ge 2$ be an integer and $V$ be a set of elements. Write $V\choose r$ for the collection of subsets of size $r$ of $V$.
An {\it $r$-uniform hypergraph}, or an {\it $r$-graph} for short, on set $V$ is a pair $H=(V,E)$, where $V$ is called {\it vertex set}, and $E\subseteq {V\choose r}$ is called {\it edge set}.
The size of an $r$-graph $H=(V,E)$ is $|E|$, denoted by $e(H)$. A {\it matching} in $H$ is a set of pairwise disjoint edges of $H$, the {\it matching number} $\nu(H)$ of $H$ is the size of a maximum matching of $H$. Given $X\subseteq V$ with $|X|=k$, the {\it $k$-degree} $d_k(X)$ of $X$  is the number of edges $e$ of $H$ with $X\subseteq e$.
The {\it maximum $k$-degree} $\Delta_{k}(H)$ of $H$ is the maximum of $d_k(X)$ over all subsets $X\subseteq V(H)$ with $|X|=k$.
We call $d_1(X)$ and $d_{r-1}(X)$ the {\it degree } and {\it codegree} of $X$ in $H$, respectively. If $r=2$, we write graph for $2$-graph for short.


The problem of bounding the size of an $r$-graph by restricting some of its parameters was well studied in literatures. In this paper, we concern the $r$-graphs with matching number $\nu$ and maximum $k$-degree $\Delta_k$, $1\le k\le r-1$.
If we only consider matching number, Erd\H{o}s and Gallai{~\cite{PT-JCTA65}} proved that: {\it For all positive integers $n$ and $\nu$, if $G$ is a graph with $n$ vertices and $\nu(G)=\nu$, then $e(G)\le\max\left\{{2\nu+1\choose 2},(n-\nu)\nu+{\nu\choose 2}\right\}$. Furthermore, the upper bound is tight.}  The $r$-graph version  of Erd\H{o}s and Gallai's Theorem for $r\ge 3$ seems more difficult than the graph version.  Erd\H{o}s~\cite{PE-PMI61} proposed the following conjecture, known as Erd\H{o}s' Matching Conjecture,
\begin{conj}\label{CONJ: c1}[Erd\H{o}s, 1965]
 Every $r$-graph $H$ on $n$ vertices with matching number $\nu(H)<s\le n/r$ satisfies
$$e(H)\le\max\left\{{{rs-1}\choose r}, {n\choose r}-{{n-s+1}\choose r}\right\}.$$
\end{conj}
Some recent improvements of Conjecture~\ref{CONJ: c1} can be found in~\cite{FRR-CPC12,PF-JCTA13,HLS-CPC12}.
If we consider the problem of finding the maximum size of an $r$-graph with given matching number and maximum $k$-degree, $1\le k\le r-1$, such a problem for  graphs has been solved by Chv\'atal and Hanson~\cite{CH-DM76} (also can be found in~\cite{BK-DM09});
for $3$-graphs, Khare~\cite{KN-DM14} provided a bound on the size of linear $3$-graphs with restricted matching number and maximum degree.  Motivated by these results, in this paper, we consider the problem for $3$-graphs and give a tight upper bound for the maximum size of $3$-graphs  on $n$ vertices  with codegree  at most $\Delta_2$ and matching number at most $\nu$.
To state our main result, we first define two functions. For positive integers $\nu, \lambda, s$, define $g(2,\lambda, s)=0$ and for $\nu>2$
{\begin{equation*}
g(\nu,\lambda,s)=
\begin{cases}
\left\lfloor\frac{\lambda}{3}{\nu\choose2}-\frac{\nu}6\right\rfloor&
\begin{split}
&\text{if $\nu\equiv0 \pmod 6$ and $\lambda\equiv1\pmod2$}\\
&\text{or $\nu\equiv2 \pmod 6$ and $\lambda\equiv1,3 \pmod6$}\\
&\text{or $\nu\equiv2 \pmod 6$, $\lambda\equiv5 \pmod6$ and $s<{\nu}/{2}$}\\
&\text{or $\nu\equiv4 \pmod6$, $\lambda\equiv1 \pmod2$ and $s<{\nu}/{2}$};
\end{split}\\
\left\lfloor\frac{\lambda}{3}{\nu\choose2}-\frac{\nu}6\right\rfloor-1&
\begin{split}
&\text{if $\nu\equiv2\pmod6$, $\lambda\equiv5\pmod6$ and $s={\nu}/{2}$}\\
&\text{or $\nu\equiv4\pmod6$, $\lambda\equiv1\pmod2$ and $s={\nu}/{2}$};\\
\end{split}\\
\left\lfloor\frac{\lambda}{3}{\nu\choose2}\right\rfloor-1&
\begin{split}
&\text{if $\nu\equiv2\pmod6$, $\lambda\equiv4\pmod6$ and $s=0$}\\
&\text{or $\nu\equiv5\pmod6$, $\lambda\equiv1\pmod3$ and $s=0$};\\
\end{split}\\
\left\lfloor\frac{\lambda}{3}{\nu\choose 2} -\frac{2s}3\right\rfloor&\text{otherwise};\\
\end{cases}
\end{equation*}}
and for given integer $n\ge \nu$, 
\begin{equation*}
f(n,\nu,\Delta_2)=\left\lfloor\frac{\nu(n-\nu)\Delta_2}{2}\right\rfloor+
\begin{cases}
g(\nu,\Delta_2,0)&\text{ if } (n-\nu)\Delta_2\text{ is even or } \nu \text{ is even};\\
g(\nu,\Delta_2,\lfloor\frac{\nu}{2}\rfloor)&\text{ otherwise}.
\end{cases}
\end{equation*}

The main result of the paper is as follows.

\begin{thm}\label{THM: MAIN}
Given positive integers $\Delta_2$ and $\nu$, if $H$ is a $3$-graph of sufficiently large order $n$ with $\Delta_2(H)\le \Delta_2$ and $\nu(H)\le\nu$, then  $e(H)\le f(n,\nu,\Delta_2)$, and  the upper bound is tight.
\end{thm}

The rest of the paper is arranged as follows. In section 2, we construct extremal 3-graphs with restricted codegree, in fact the constructions come  from combinatorial design. In section $3$, we establish a structural lemma. Then the proof of Theorem~\ref{THM: MAIN} is presented in Section $4$, and we end the paper with some discussions.

\section{ Extremal 3-graphs with restricted codegree }

A {\it partial triple system }(PTS for short) PTS$(\nu,\lambda)$ is a  set $V$ of $\nu$ elements and a collection $\mathcal{B}$ of triples of $V$ so that every pair of elements occurs in at most $\lambda$ triples of $\mathcal{B}$.
If every pair of elements of $V$ occurs in exactly $\lambda$ triples of $\mathcal{B}$, then such a PTS$(\nu,\lambda)$ is called a {\it triple system} TS$(\nu,\lambda)$. 
A TS$(\nu,1)$ is called a {\it Steiner triple system} STS$(v)$.
A {\it maximum partial triple system} MPTS$(\nu,\lambda)$ is a PTS$(\nu, \lambda)$ with triples as many as possible.
A set system $(V,\mathcal{B})$ is called a {\it pairwise balanced design PBD$(\nu,\{3,5^*\},\lambda)$} if the following properties are satisfied:
(1) $|V|=\nu$, (2) $|B|=3$ for all $B\in \mathcal{B}$ except a specific one $B_0$ with $|B_0|=5$, and (3) every pair of points is contained in exactly $\lambda$ blocks of $\mathcal{B}$.

\noindent{\bf Remark:} It is natural to see a PTS$(\nu,\lambda)$ as a 3-graph with vertex set $V$ and edge set $\mathcal{B}$ such that $\Delta_2(\mbox{PTS}(\nu,\lambda))\le \lambda$. So, in this paper, we do not distinguish a 3-graph and a PTS.

\begin{thm}\label{THM: T1,T2,T3}
Let $\nu$ and $\lambda$ be positive integers. We have

(1)(\cite{CD-HCD07}) A TS$(\nu,\lambda)$ exists if and only if $\nu\neq 2$ and $\lambda\equiv 0\pmod {\gcd(\nu-2,6)}$.

(2) (\cite{K-CDM47}) A STS$(\nu)$ exists if and only if $\nu\equiv 1,3\pmod 6$.

(3) (\cite{RS-JCTA87})
The number of triples in an $MPTS(\nu,\lambda)$ is $g(\nu,\lambda,0)$.

(4) (\cite{CD-HCD07})
{A PBD$(\nu,\{3,5^*\},1)$ with $\nu>5$ exists if and only if $\nu\equiv 5\pmod 6$.}
\end{thm}

The {\it leave} of a PTS$(\nu, \lambda)$ is the graph $G$ on vertex set $V$ in which an edge $\{x,y\}$ appears $\lambda-s$ times in $G$ if and only if  $\{x,y\}$ is contained in exactly $s$ triples of $\mathcal{B}$.
Clearly,
\begin{equation}\label{EQN: e1-leave-size}
 {e(PTS(\nu, \lambda))}=\frac{\lambda {\nu\choose 2}-{e(G)}}3.
\end{equation}
Let $G_1\cup G_2$ (resp. $G_1\dot\cup G_2$) denote the union (resp. disjoint union) of two graphs $G_1$ and $G_2$. For integer $k$, let $kG$ be the graph obtained from $G$ by replacing each edge of $G$ by $k$ copies of the edge, and write $k\{G\}$ for the disjoint union of $k$ copies of $G$.
When can a graph $G$ be the leave of a PTS? The following lemma gives us an obvious necessary condition and a sufficient and necessary condition depending on $G$.

\begin{lem}\label{LEM: L4,T5}
(1) (Lemma 2.1 in \cite{CR-LEN91}) If a multigraph $G$ with $\nu$ vertices and $m$ edges is the leave of a $PTS(\nu,\lambda)$, then $d_G(v)\equiv\lambda(\nu-1)\,(\text{mod }2)$ for $v\in V(G)$, and $m\equiv\lambda\binom{\nu}{2}\,(\text{mod }3)$.


(2) (Theorem 7 in~\cite{CR-DM13}, also can be found in \cite{CR-GC86,CR-JG87})
Let $\nu>2$ and let $G$ be a simple graph on at most $\nu$ vertices with every vertex of degree $0$ or $2$, then $G$ is the leave of a PTS$(\nu,\lambda)$ if and only if
\begin{itemize}
\item either $\lambda$ is even or $\nu$ is odd,

\item $3$ divides $\lambda\binom{\nu}{2}-{e(G)}$, and

\item (a) if $\lambda=1$ and $\nu=7$, then $G\neq 2\{C_3\}$,

   (b) if $\lambda=1$ and $\nu=9$, then $G\neq C_4\dot\cup C_5$, and

   (c) if $\lambda=2$ and $\nu=6$, then $G\neq 2\{C_3\}$.
   
\end{itemize}
\end{lem}

Let $s$ be a nonnegative integer.
A PTS$(\nu,\lambda,s)$ is a PTS whose leave contains at least $s$ independent edges.
Let MPTS$(\nu,\lambda,s)$ be a PTS$(\nu,\lambda, s)$ with the maximum number of triples among all of such PTSs.
Clearly, the leave of an MPTS$(\nu,\lambda, s)$ has the fewest edges among all of the leaves of PTS$(\nu,\lambda,s)$.  So, in order to determine the size of  an MPTS$(\nu,\lambda,s)$, we will look for the minimum possible leaves of MPTS$(\nu, \lambda, s)$.
The following theorem is a generalization of (3) of Theorem~\ref{THM: T1,T2,T3}, which has independent interest in combinatorial design.
The proof is similar with the one of (3) of Theorem~\ref{THM: T1,T2,T3},  we give the proof here for completeness.

\begin{thm}\label{THM: T6}
The number of triples in an MPTS$(\nu,\lambda,s)$ is $g(\nu,\lambda,s)$.
\end{thm}
\begin{proof} The (3) of Theorem~\ref{THM: T1,T2,T3} implies that the theorem holds for $s=0$. So in the following we assume that $s>0$ and we shall distinguish two cases according to the value of $\lambda$.

\noindent{\bf Case 1.} $1\leq\lambda\leq \gcd(\nu-2,6)$.


\renewcommand{\arraystretch}{2}
\begin{table}[h]

  \centering
  \fontsize{12}{8}\selectfont
  \caption{Possible values of $\lambda$ and $\nu$
  }\label{tab:}
    \begin{tabular}{|c|c|c|c|c|c|c|}
    \hline
    \multirow{2}{*}{$\lambda$}&
    \multicolumn{6}{c|}{$\nu$ under modulo 6}\cr\cline{2-7}
    &0&1&2&3&4&5\cr
    \hline
    1&odd/0&even/0&odd/1&even/0&odd/0&even/1\cr\hline
    2&even/0&&even/2&&even/0&even/2\cr\hline
    3&&&odd/0&&&even/0\cr\hline
    4&&&even/1&&&\cr\hline
    5&&&odd/2&&&\cr\hline
    6&&&even/0&&&\cr\hline
    \end{tabular}
\begin{itemize}
\item \small Odd or even represents the parity of the degrees of vertices in the leave of an MPTS$(\nu,\lambda,s)$.
\item \small The number under ``/" denotes the number of edges of the leave of an MPTS$(\nu,\lambda,s)$ under modulo 3.
\item\small The blank space in position $(i,j)$ means that the situation $\lambda=i$ and $\nu=j\pmod 6$ can not happen.
\end{itemize}
\end{table}
Let $G$ be the leave of an MPTS$(\nu,\lambda,s)$. By (1) of Lemma~\ref{LEM: L4,T5}, $d_G(v)\equiv \lambda(\nu-1)\pmod 2$ { for $v\in V(G)$} and $e(G)\equiv \lambda{\nu\choose 2}\pmod 3$. All possible values of $\lambda$ and $\nu$ are listed in Table~\ref{tab:}. We claim that $G$ actually exists in each case.


(I). If $\nu\equiv 0,2\pmod 6$ and $\lambda=1$, then $d_G(v)\equiv 1\pmod 2$ for all $v\in V(G)$ and $e(G)\equiv 0$ or $1\pmod 3$.
 This implies that $e(G)\ge \frac{\nu}{2}$ and the equality holds if and only if $G\cong\frac{\nu}{2}\{K_2\}$. Such a leave can be realized by a PTS$(\nu,1,s)$ obtained from a STS$(\nu+1)$ by removing a fixed element (i.e. removing all triples containing the fixed element), where the existence of STS$(\nu+1)$ is guaranteed by Theorem~\ref{THM: T1,T2,T3}. In fact, $G$ consists of $\frac\nu 2$ independent edges, each of which corresponds to a triple removed from STS$(\nu+1)$.
 This also implies that $e(\mbox{MPTS}(\nu,1,s))=\frac 13{\nu+1\choose 2}-\frac\nu 2=\frac{1}{3}\left({ \binom{\nu}{2}}-\frac{\nu}{2}\right)=g(\nu,1,s)$.

(II).  If $\nu\equiv 0,4 \pmod6$ and $\lambda=2$, or $\nu\equiv 2\pmod6$ and $\lambda=6$, or $\nu\equiv 1,3\pmod 6$ and $\lambda=1$,  or $\nu\equiv 5\pmod6$ and $\lambda=3$, then $d_G(v)\equiv 0\pmod 2$ for all $v\in V(G)$ and $e(G)\equiv 0\pmod3$ and so we have $e(G)\ge 2s+\sigma$, where $\sigma\in\{0,1,2\}$ and $s\equiv \sigma\pmod3$.
By (2) of Lemma~\ref{LEM: L4,T5}, there exists a PTS$(\nu,\lambda, s)$ with $G=C_{2s+\sigma}$ as its leave, unless $s=\lfloor\frac{\nu}{2}\rfloor$ and $\nu\equiv 2,4,5 \pmod6$.
Since $e(G)$ achieves its lower bound, we have $e($MPTS$(\nu,\lambda,s))=\frac{1}{3}({\lambda \binom{\nu}{2}}-2s-\sigma)=g(\nu,\lambda,s)$.
The proof for the cases $s=\lfloor\frac{\nu}{2}\rfloor$ and $\nu\equiv 2,4,5 \pmod6$ will be included in (XI), { (IV)} and (VII), respectively.


(III). If $\nu\equiv 4 \pmod 6$ and $\lambda=1$, then $d_G(v)\equiv 1\pmod 2$ for all $v\in V(G)$ and $e(G)\equiv 0\pmod 3$. This implies that $e(G)\ge \nu/2+1$,
and, moreover, if we require $\nu(G)=\nu/2$ then $e(G)\ge \nu/2+4$. We claim that there do exist PTS$(\nu,1,s)$ such that the { size} of its leave achieves the lower bound $\nu/2+1$
for $s<\nu/2$, and the lower bound $\nu/2+4$
for $s=\nu/2$. Choose a PBD$(\nu+1,\{3,5^*\},1)$
(the  existence of such a design is guaranteed by Theorem~\ref{THM: T1,T2,T3}), and denote the specific block of size five by $\{u,w,x,y,z\}$.  Let  PTS$(\nu+1,1)$ be a PTS obtained from the PBD$(\nu+1,\{3,5^*\},1)$ by replacing $\{u,w,x,y,z\}$ by two triples $\{u,w,z\}$ and $\{u,x,y\}$.
If $s\leq\frac{\nu}{2}-1$, delete $x$ from the PTS$(\nu+1,1)$, then the resulting PTS is a PTS$(\nu,1,s)$ with its leave isomorphic to  $K_{1,3}\dot\cup\frac{\nu-4}{2}\{K_2\}$, as claimed.
So {$e($MPTS$(\nu,1,s))=\frac{1}{3}\left({ \binom{\nu}{2}}-\frac{\nu}{2}-1\right)=g(\nu,1,s)$.}  For $s=\frac{\nu}{2}$, eliminate $u$ from the PTS$(\nu+1,1)$, the resulting PTS is a PTS$(\nu,1,\frac{\nu}{2})$ with its leave isomorphic to $K_4\dot\cup\frac{\nu-4}{2}\{K_2\}$, as claimed.
So {$e($MPTS$(\nu,1,\frac{\nu}{2}))=\frac{1}{3}\left({\binom{\nu}{2}}-\frac{\nu}{2}-4\right)=g(\nu,1,s)$. }

(IV). For $\nu\equiv 4\pmod 6$, $\lambda=2$ and $s={\nu}/{2}$, the possible $G$ has the property that $d_G(v)\equiv 0\pmod 2$, $e(G)\equiv 0\pmod3$ and $\nu(G)=\nu/2$. This implies that $e(G)\ge\nu+2$. By (III), there exists an MPTS$(\nu,1,\nu/2)$ with its leave isomorphic to $K_4\dot\cup\frac{\nu-4}{2}\{K_2\}$. Denote $V(K_4)=\{x,y,z,w\}$. Define
 $$\mbox{PTS}(\nu,2,\nu/2)=\mbox{MPTS}(\nu,1,\nu/2)\cup\mbox{MPTS}(\nu,1,\nu/2)\cup\{x,y,z\}\cup\{x,y,w\}.$$
 It can be easily checked that this is indeed a PTS$(\nu,2,\nu/2)$ and its leave $G=M\dot\cup\frac{\nu-4}{2}\{2K_2\}$, where $M$ is the join graph of $\overline{K_2}$ (spanned by $\{x,y\}$) and $2K_2$ (spanned by $\{w,z\}$).
Since $e(G)=\nu+2$ achieves its lower bound,  we have $e($MPTS$(\nu,2,{\nu}/{2}))=\frac{1}{3}\left({\lambda \binom{\nu}{2}}-\nu-2\right)=g(\nu,2,s)$.

(V). If $\nu\equiv 5\pmod 6$ and $\lambda=1$ or $\nu\equiv 2\pmod 6$ and $\lambda=4$, then $d_G(v)\equiv0\pmod2$ for all  $v\in V(G)$ and $e(G)\equiv1\pmod3$. This implies that $e(G)\ge 2s+\sigma$, where $\sigma\in\{0,1,2\}$ and $s\equiv \sigma+2\pmod 3$. By Lemma~\ref{LEM: L4,T5}, there is a PTS$(\nu,\lambda, s)$ with its leave $G$ isomorphic to $C_{2s+\sigma}$,
unless $s={\nu}/{2}$ when $\nu\equiv 2\pmod6$ (which will be included in (IX)). Clearly, $e(G)$ attains its lower bound, so we have $e($MPTS$(\nu,\lambda,s))
=\frac{1}{3}\left({\lambda \binom{\nu}{2}}-2s-\sigma\right)=g(\nu,\lambda,s)$.

(VI). If $\nu\equiv 2,5\pmod6$ and $\lambda=2$, then $d_G(v)\equiv0\pmod2$ for all $v\in V(G)$ and $e(G)\equiv2\pmod3$. This implies that $e(G)\ge 2s+\sigma$, where $\sigma\in\{0,1,2\}$ and $s\equiv \sigma+1\pmod3$. By (2) of Lemma~\ref{LEM: L4,T5}, there is a {PTS$(\nu,2)$}  with its leave $G$ isomorphic to $C_{2s+\sigma}$, unless $s=1$.
So, if $s\not=1$  we have $e($MPTS$(\nu,2,s))=\frac{1}{3}\left({\lambda \binom{\nu}{2}}-2s-\sigma\right)=g(\nu,2,s)$. For $s=1$, we claim that there is a PTS$(\nu, 2, 1)$ with its leave isomorphic to $2K_2$ (note that $2K_2$ is not simple, we can not apply (2) of Lemma~\ref{LEM: L4,T5} here).
For $\nu\equiv 2\pmod 6$, to construct a PTS$(\nu,2,1)$, we first choose a STS$(\nu+1)$ (the existence is guaranteed by Theorem~\ref{THM: T1,T2,T3}) and a triple $\{x,y,z\}$ from it, then replace $z$ by $y$ in each triple containing $z$ except $\{x,y,z\}$ and duplicate each triple containing no $y$, finally delete $z$, the resulting {PTS} is a PTS$(\nu,2,1)$ with leave $2K_2$ (spanned by $\{x,y\}$), as desired.
So, $e($MPTS$(\nu,2,1))=\frac{1}{3}\left({2\binom{\nu}{2}}-2\right)=g(\nu,2,1)$. For $\nu\equiv 5\pmod6$,
take the union of two MPTS$(\nu,1,1)$ with leaves $G_1=xyzwx$ and $G_2=xywzx$, respectively, (the existence of such PTSs is  guaranteed by (V))
and two triples $\{x,y,z\}$ and $\{x,y,w\}$, the resulting  { PTS} is a PTS$(\nu, 2,1)$ with its leave isomorphic to $2K_2$ (spanned by $\{z,w\}$). So
$e($MPTS$(\nu,2,1))=\frac{1}{3}\left({2\binom{\nu}{2}}-2\right)=g(\nu,2,1)$.

(VII). If $\nu\equiv 5\pmod 6$, $\lambda=3$ and $s=\lfloor{\nu}/{2}\rfloor$, then $e(G)\ge 2s+2=\nu+1$ by (II).
 Construct a PTS by taking the union of an MPTS$(\nu,1,\lfloor\frac{\nu}{2}\rfloor)$ with its leave isomorphic to $C_{\nu-1}$ (the existence is guaranteed by (V), and assume $\{x\}=V\setminus V(C_{\nu-1})$)  and an MPTS$(\nu,2,1)$ with its leave isomorphic to $2K_2$ (the existence is guaranteed by (VI), and assume $V(2K_2)=\{x,y\}$). The resulting PTS is actually a PTS$(\nu,3,\lfloor\frac{\nu}{2}\rfloor)$ with its leave $G$ isomorphic to $C_{\nu-1}\cup2K_2$, where $C_{\nu-1}$ and $2K_2$ has a unique common vertex $y$. Since $e(G)$ attains its lower bound, we have $e($MPTS$(\nu,3,\lfloor{\nu}/{2}\rfloor))=\frac{1}{3}({3\binom{\nu}{2}}-\nu-1)=g(\nu,3,\lfloor\nu/2\rfloor)$.

(VIII). If $\nu\equiv 2\pmod 6$ and $\lambda=3$, then $d_G(v)\equiv 1\pmod2$ and $e(G)\equiv 0\pmod 3$.  This implies that $e(G)\ge {\nu}/{2}+2$. Take the union of an MPTS$(\nu,1,s)$ with leave $G_1=\frac\nu2\{K_2\}$ (guaranteed by (I)) such that $xy\in E({ G_1})$ and an MPTS$(\nu,2,1)$ with leave $2K_2$ spanned by $\{x,y\}$ (guaranteed by (VI)), the resulting { PTS} is a PTS$(\nu,3,s)$ with leave $G=G_1\cup 2K_2\cong 3K_2\dot\cup\frac{\nu-2}{2}\{K_2\}$.
Clearly, $e(G)$ obtains its lower bound, so  {$e($MPTS$(\nu,3,s))=\frac{1}{3}\left({ 3\binom{\nu}{2}}-{\nu}/{2}-2\right)=g(\nu,3,s)$}.

(IX). If $\nu\equiv 2\pmod 6$, $\lambda=4$ and $s=\frac{\nu}{2}$, then $e(G)\ge \nu+2$ by (V).
Take the union of an MPTS$(\nu,1,s)$ with leave $G_1=\frac\nu2\{K_2\}$ (guaranteed by (I))  and an MPTS$(\nu,3,s)$ with leave $G_2=3K_2\dot\cup\frac{\nu-2}{2}\{K_2\}$ (guaranteed by (VIII)), the resulting { PTS} is a PTS$(\nu,4,{\nu}/{2})$ with leave $G_1\cup G_2\cong4K_2\dot\cup\frac{\nu-2}{2}\{2K_2\}$.
Since $e(G)$ attains its lower bound, we have $e($MPTS$(\nu,4,{\nu}/{2}))=\frac{1}{3}\left({4 \binom{\nu}{2}}-\nu-2\right)=g(\nu,4,\nu/2)$.

(X). If $\nu\equiv 2\pmod 6$ and $\lambda=5$, then $d_G(v)\equiv1\pmod2$ for all $v\in V(G)$ and $e(G)\equiv 2\pmod 3$. This implies that $e(G)\ge \nu/2+1$ and the equality holds if and only if $G=K_{1,3}\dot\cup\frac{\nu-4}2\{K_2\}$, moreover, if we require $\nu(G)=\nu/2$ then $e(G)\ge \nu/2+4$. Hence, for $s\leq{\nu}/{2}-1$, we first  take the union of two MPTS$(\nu,2,1)$ with leaves $2K_2$ (guaranteed by (VI)) such that the leaves  are spanned by $\{x,z\}$ and $\{y,z\}$, respectively,  the resulting  { PTS} is a PTS$(\nu,4, 1)$, then take the union of the PTS$(\nu,4,1)$ and an MPTS$(\nu,1,s)$ with leave $G_1=\frac\nu2\{K_2\}$ (guaranteed by (I)) such that $xy,zw\in E( {G_1})$
and finally union the triple $\{x,y,z\}$,
the resulting { PTS} is a PTS$(\nu,5,s)$ with leave $G$ isomorphic to $K_{1,3}\dot\cup\frac{\nu-4}{2}\{K_2\}$, where $K_{1,3}$ is a star spanned by $\{x,y,z,w\}$ with $z$ as its center.
Since $e(G)$ attains its lower bound, we have $e($MPTS$(\nu,5,s))=\frac{1}{3}\left({5 \binom{\nu}{2}}-\frac{\nu}{2}-1\right)=g(\nu, 5, s)$. For $s={\nu}/{2}$, take the union of an MPTS$(\nu,1,s)$ with leave $G_1=\frac\nu2\{K_2\}$ (guaranteed by (I)) such that $xy,zw\in E(G_1)$ and two MPTS$(\nu,2, 1)$ with leaves  isomorphic to $2K_2$ spanned by $\{x,y\}$ and $\{z,w\}$, respectively, (the existence is guaranteed by (VI)), the resulting PTS is a PTS$(\nu,5,\nu/2)$ with leave $G$ isomorphic to $2\{3K_2\}\dot\cup\frac{\nu-4}{2}\{K_2\}$.
Since $e(G)$ obtains its lower bound, we have $e($MPTS$(\nu,5,\frac{\nu}{2}))=\frac{1}{3}\left({5\binom{\nu}{2}}-\frac{\nu}{2}-4\right)=g(\nu,5,\nu/2)$.

(XI). If $\nu\equiv 2\pmod6$, $\lambda=6$ and $s={\nu}/{2}$, then $e(G)\ge \nu+1$ by (II).
 Take the union of an MPTS$(\nu,2,\frac{\nu}{2})$ with leave $G_1= {C_\nu}$ (guaranteed by (VI)) such that $xy,yz\in E(G_1)$ and two MPTS$(\nu,2,1)$ with leaves isomorphic to $2K_2$ (guaranteed by (VI) and assume that the leaves spanned by $\{x,z\}$ and $\{y,w\}$, respectively) and a triple $\{x,y,z\}$, the resulting PTS is a PTS$(\nu,6,{\nu}/{2})$ with leave $G=2K_2\cup C_{\nu-1}$, where $2K_2$ is spanned by $\{y,w\}$ and $C_{\nu-1}$ is spanned by $V(G)\setminus\{y\}$. Clearly, $\nu(G)=\nu/2$ and $e(G)$ achieves its lower bound, we have $e($MPTS$(\nu,6,{\nu}/{2}))=\frac{1}{3}\left({6 \binom{\nu}{2}}-\nu-1\right)=g(\nu,6,\nu/2)$.

Until now, all the possible values of $\nu$ and $1\leq\lambda\leq \gcd(\nu-2,6)$ have been discussed, this completes the proof of Case 1.

\noindent{\bf Case 2}. $\lambda\geq \gcd(\nu-2,6)+1$.

 Let $t\equiv\lambda-1\pmod{\gcd(\nu-2,6)}$. So the union of a TS$(\nu,\lambda-1-t)$ (the existence is guaranteed by Theorem~\ref{THM: T1,T2,T3}) and an MPTS$(\nu,t+1,s)$ is a PTS$(\nu,\lambda,s)$.
 Moreover, such a PTS$(\nu,\lambda,s)$ has the same leave as MPTS$(\nu,t+1,s)$.
 We claim that such a PTS is an MPTS$(\nu,\lambda,s)$, that is  the leave of such a PTS has minimum size. Let $G$ be the leave of such a PTS$(\nu,\lambda,s)$. The existence of $G$ is guaranteed by Case 1. Let $Q$ be a PTS$(\nu,\lambda,s)$ with its leave $G'$. Then $d_{G'}(v)\equiv d_G(v)\pmod2$ for any $v\in V(G')(=V(G))$ and $e(G')\equiv e(G)\pmod3$. By Case 1, $e(G)$ achieves its lower bound. This implies our claim.
 So $e($MPTS$(\nu,\lambda,s))=\frac{1}{3}\left({\lambda \binom{\nu}{2}}-e(G)\right)=g(\nu,\lambda,s)$.
\end{proof}

\section{A Configuration Lemma}

In this section, we prove a useful lemma.

Given a $3$-graph $H=(V,E)$ and a vertex $x\in V$, the {\it link graph} of $x$, denoted by $L_H(x)$ (or $L(x)$ for short), is a graph with vertex set $V$ and edge set $\{ S : S\subseteq {V\choose 2} \mbox{ and } S\cup\{x\}\in E\}$.
Given a graph $G$, let $f$ be a non-negative, integer-valued function on $V(G)$, a spanning subgraph $H$ of $G$ is called an {\it $f$-factor} of $G$ if $d_H(v)=f(v)$ for any $v\in V(G)$.

\begin{thm}[\cite{KT-NZ00}]\label{THM: T0}
Let $G$ be a graph and $a\leq b$, two positive integers. Suppose further that

1. $\delta(G)\geq \frac{b}{a+b}|V(G)|$, and

2. $|V(G)|>\frac{a+b}{a}(b+a-3)$.\\
If $f$ is a function from $V(G)$ to $\{a,a+1,\cdots,b\}$ such that $\sum\limits_{v\in V(G)}f(v)$ is even, then $G$ has an $f$-factor.
\end{thm}

\begin{lem}\label{LEM: L7}
Let $\nu$ and $\Delta_2$ be positive integers and $n$ is appropriately large relative to $\nu$ and $\Delta_2$. Let $H$ be a  $3$-graph on $n$ vertices with $\Delta_2(H)\le \Delta_2$. If each edge of $H$ intersects a given subset of $V(H)$ of size $\nu$, then $e(H)\le f(n,\nu,\Delta_2)$.
\end{lem}

\begin{proof}[Proof]
Let $V_0\subseteq V(H)$ such that $|V_0|=\nu$ and $e\cap V_0\not=\emptyset$ for any $e\in E(H)$, and let $V_1=V(H)\backslash V_0$. Let $$E_i=\{e\in E(H)\,:\, |e\cap V_0|=i\}$$ and $\varepsilon_i=|E_i|$. By assumption, $E(H)=E_1\cup E_2\cup E_3$. Hence
\begin{gather*}
e(H)=\varepsilon_1+\varepsilon_2+\varepsilon_3.\nonumber
\end{gather*}
Since $\Delta_2(H)\le \Delta_2$, we have
\begin{gather}
3\varepsilon_3+\varepsilon_2\leq\binom{\nu}{2}\Delta_2,\label{EQN: e_2-varepsilon_3}\\
2\varepsilon_2+2\varepsilon_1\leq \nu(n-\nu)\Delta_2.\label{EQN:e_3-varepsilon_1+2}
\end{gather}
{Denote $V_0=\{v_i: 1\leq i\leq \nu\}$.}



\noindent{\bf Case 1}. $(n-\nu)\Delta_2$ is even.

By {Theorem~\ref{THM: T6}}, we have
\begin{equation}\label{EQN: varepsilon_3}
\varepsilon_3\le g(\nu,\Delta_2,0),
\end{equation}
and the upper bound can be achieved by choosing $E_3$ as an MPTS$(\nu, \Delta_2, 0)$.

Therefore, by (\ref{EQN:e_3-varepsilon_1+2}) and (\ref{EQN: varepsilon_3}),
\begin{equation}\label{EQN: varepsilon_1+2+3(0)}
\varepsilon_1+\varepsilon_2+\varepsilon_3\le\frac{\nu(n-\nu)\Delta_2}2+  g(\nu,\Delta_2,0)=f(n,\nu,\Delta_2).
\end{equation}
So, to  show our result, it is sufficient to verify that the upper bound of (\ref{EQN: varepsilon_1+2+3(0)}) is tight.
Define a function $f: V_1\rightarrow \{\Delta_2\}$,  as $n$ is sufficiently large, Theorem~\ref{THM: T0} implies that we can find $\nu$ edge-disjoint $f$-factors $G_1, G_2,\ldots, G_\nu$ of $K_{n-\nu}$. Clearly, each $f$-factor $G_i$ is a $\Delta_2$-regular graph on $V_1$. 
Define $E_1=\cup_{i=1}^\nu\{e\cup\{v_i\} : e\in E(G_i)\}$. Then $\varepsilon_1=\frac{\nu\Delta_2(n-\nu)}{2}$. Hence if we take $E_2=\emptyset$ and $E_3$ as an MPTS$(\nu, \Delta_2, 0)$, then $E_1\cup E_2\cup E_3$ is a desired edge set of a 3-graph $H$ with $e(H)$ achieving the upper bound.

\noindent{\bf Case 2.} $\Delta_2(n-\nu)$ is odd.


We first define two 3-graphs $H_1$ and $H_2$ which will be used in this case. Define a function $f: V_1\mapsto\{\Delta_2-1, \Delta_2\}$. Since $\Delta_2(n-\nu)$ is odd, both of $\Delta_2$ and $n-\nu$ are odd. Hence $(n-\nu)\Delta_2-1$ and $(n-\nu)\Delta_2-3$ are even. Note that $n$ is sufficiently large. By Theorem~\ref{THM: T0}, we can find $\nu$ edge-disjoint $f$-factors $G_1, G_2,\ldots, G_\nu$ in $K_{n-\nu}$ such that each $G_i$ has exactly one vertex $y_i$ with $d_{G_i}(y_i)=\Delta_2-1$, and we also can find $\nu$ edge-disjoint $f$-factors $G_1, G_2,\ldots, G_\nu$ in $K_{n-\nu}$ such that each $G_i$ for $i<\nu$ has exactly one vertex $y_i$ with $d_{G_i}(y_i)=\Delta_2-1$ and $G_\nu$ has three vertices $y^t_\nu,t=1,2,3$ with $d_{G_\nu}(y^t_\nu)=\Delta_2-1$. Moreover, we may assume $y_{2i-1}=y_{2i}$ for $1\le i\le \lfloor\frac \nu 2\rfloor$ for the first case, and assume $y_{2i-1}=y_{2i}$ for $1\le i\le \lfloor\frac \nu 2\rfloor-2$ 
for the latter case.
Denote $\ell=\lfloor\frac\nu2\rfloor$. Let 
{$$E_1'=\bigcup_{i=1}^\nu\{e\cup\{v_i\} : e\in E(G_i)\},$$}
{$$E_2'=\bigcup_{i=1}^{\ell}\{v_{2i-1}, v_{2i}, y_{2i}\}$$} for the first case and
{$$E_2''=\left(\bigcup_{i=1}^{\ell-2}\{v_{2i-1}, v_{2i}, y_{2i}\}\right)\bigcup\left(\bigcup_{t=1}^3\{v_{\nu-t}, v_\nu, y^t_\nu\}\right)$$} for the latter case.
Define $H_1=(V_0\cup V_1, E_1'\cup E_2')$ and $H_2=(V_0\cup V_1, E_1'\cup E_2'')$.

\noindent {\bf Subcase 2.1.} $\Delta_2(\nu-1)$ is odd.

Since $\Delta_2(\nu-1)$ is odd, we have $\nu$ is even. We show that the upper bound of (\ref{EQN: varepsilon_1+2+3(0)}) is also tight in this case.
Choose $E_3$ as an MPTS$(\nu,\Delta_2,0)$ and $L$ as the leave of $E_3$. By Lemma~\ref{LEM: L4,T5}, $d_L(v)\equiv 1\pmod2$. By $3\varepsilon_3+e(L)={\nu\choose 2}\Delta_2$, we have $e(L)={\nu\choose 2}\Delta_2-3g(\nu,\Delta_2,0)=\nu/2$, $\nu/2+1$ or $\nu/2+2$. By the proof of Theorem~\ref{THM: T6}, we may choose $E_3$ such that $\nu(L)=\nu/2$ when {$e(L)\in\{\nu/2, \nu/2+2\}$} and {$L=K_{1,3}\dot\cup\frac{\nu-4}2\{K_2\}$} when $e(L)=\nu/2+1$. Now we partition $V_0$ into sets 
{$\{v_1,v_2\}, \{v_3, v_4\},\ldots, \{v_{\nu-1},v_\nu\}$} such that 
{$\{v_1v_2, v_3v_4, \ldots, v_{\nu-1}v_\nu\}$} forms a perfect matching of $L$ when $\nu(L)=\nu/2$,
and {$\{v_1v_2,  \ldots, v_{\nu-3}v_{\nu}\}$} forms a maximum matching of $L$ and 
{$\{v_{\nu-3}, v_{\nu-2}, v_{\nu-1}, v_\nu\}$} induces a $K_{1,3}$ with $d_L(v_{\nu})=3$ {when $e(L)=\nu/2+1$}. If $\nu(L)=\nu/2$, then we may
choose $E_1=E_1'$ and $E_2=E_2'$. If 
{$e(L)=\nu/2+1$}, then we
{ choose $E_1=E_1'$ and $E_2=E_2''$}. Hence
$\varepsilon_1+\varepsilon_2=\frac{\nu(n-\nu)\Delta_2}{2}$ in each case. Therefore,
$$\varepsilon_1+\varepsilon_2+\varepsilon_3=\frac{\nu(n-\nu)\Delta_2}{2}+g(\nu,\Delta_2,0),$$ as claimed.




\noindent{\bf Subcase 2.2.} $\Delta_2(\nu-1)$ is even.

Since $\Delta_2(\nu-1)$ is even, we have $\nu$ is odd.
Let $L$ be the graph with vertex set $V_0$ and edge set $\{e\cap V_0 : e\in E_2\}$.
Denote 
$d_i=d_L(v_i)$. Then $\sum_{i=1}^\nu d_i=2\varepsilon_2$.
Set $p=|\{i: d_i\text{ is odd}\}|$ and $q=|\{i: d_i\text{ is even}\}|$.
For each $v_i\in V_0$, let $L_i$ be the graph with vertex set $V_1$ and edge set $\{e\cap V_1 : e\in E_1 \mbox{ and } v_i\in e\}$.
Then $\sum_{i=1}^\nu e(L_i)=\varepsilon_1$. Since $d_i+\sum\limits_{v\in V(L_i)}d_{L_i}(v)=\sum\limits_{\substack{S=\{v_i,v\}\\ v\in V_1}}d_H(S)\le (n-\nu)\Delta_2$, we have $e(L_i)\le \left\lfloor\frac{\Delta_2(n-\nu)-d_i}{2}\right\rfloor$. So
\begin{eqnarray}\label{EQN: e_1+e_2(1)}
\varepsilon_1+\varepsilon_2&\leq&\sum_{d_i\text{ odd}}\frac{(n-\nu)\Delta_2-d_i}{2}+\sum_{d_i\text{ even}}\frac{(n-\nu)\Delta_2-d_i-1}{2}+\varepsilon_2\nonumber\\
   &=&\frac{\nu(n-\nu)\Delta_2}{2}-\frac{\text{$q$}}{2}\nonumber\\
   &=&\frac{\nu\Delta_2(n-\nu)-1}{2}-\frac{\text{$q$}-1}{2}.
\end{eqnarray}

Clearly, if we see $E_3$ as a PTS$(\nu,\Delta_2)$ on $V_0$, then $L$ is a subgraph of the leave $G$ of $E_3$. By Lemma~\ref{LEM: L4,T5}, $d_G(v)\equiv 0\pmod 2$ { for all $v\in V(G)$} and $e(G)\equiv {\nu\choose 2}\Delta_2\pmod3$. Since $L$ has $p$ vertices of degree odd, we have $e(G)-e(L)\ge p/2$.
Since $3\varepsilon_3+e(G)={\nu\choose 2}\Delta_2$, we have
\begin{eqnarray*}
\varepsilon_3\leq\left\lfloor\frac{1}{3}\left(\Delta_2\binom{\nu}{2}-\varepsilon_2-\frac{p}{2}\right)\right\rfloor.
\end{eqnarray*}
Note that $p+q=\nu$ and $q\ge 1$. We have,
$$\varepsilon_1+\varepsilon_2+\varepsilon_3\leq\frac{\Delta_2\nu(n-\nu)-1}{2}+\left\lfloor\frac{1}{3}\left(\Delta_2\binom{\nu}{2}-\varepsilon_2-\frac{\nu-1}2-q+1\right)\right\rfloor.$$
If $\varepsilon_2\geq \lfloor\frac{\nu}{2}\rfloor$, then
\begin{eqnarray*}
\varepsilon_1+\varepsilon_2+\varepsilon_3&&\leq\frac{\Delta_2\nu(n-\nu)-1}{2}+\left\lfloor\frac{1}{3}\left(\Delta_2\binom{\nu}{2}-\nu-q+2\right)\right\rfloor\\
     &&\leq\frac{\Delta_2\nu(n-\nu)-1}{2}+g(\nu,\Delta_2,\left\lfloor{\nu}/{2}\right\rfloor).
\end{eqnarray*}
If $\varepsilon_2<\lfloor{\nu}/{2}\rfloor$, note that $\nu=p+q\le 2\varepsilon_2+q$, we have $q\geq \nu-2\varepsilon_2$. So
\begin{eqnarray*}
\varepsilon_1+\varepsilon_2+\varepsilon_3&&\leq\frac{\Delta_2\nu(n-\nu)-1}{2}+\left\lfloor\frac{1}{3}\left(\Delta_2\binom{\nu}{2}-\nu-(\lfloor{\nu}/{2}\rfloor-\varepsilon_2-1)\right)\right\rfloor\\
     &&\leq\frac{\Delta_2\nu(n-\nu)-1}{2}+g(\nu,\Delta_2,\lfloor{\nu}/{2}\rfloor).
\end{eqnarray*}
Therefore, in each case,  we have
\begin{equation}\label{EQN: varepsilon_1+2+3(1)}
e(H)\leq\left\lfloor\frac{\nu\Delta_2(n-\nu)}{2}\right\rfloor+g(\nu,\Delta_2,\lfloor{\nu}/{2}\rfloor)=f(n,\nu,\Delta_2).
\end{equation}
Now we construct $3$-graphs $H$ such that $e(H)$ attains the upper bound of (\ref{EQN: varepsilon_1+2+3(1)}). Choose $E_3$ as an MPTS$(\nu,\Delta_2,\lfloor\nu/2\rfloor)$ and assume $G$ is the leave of $E_3$. So $\nu(G)=\lfloor\nu/2\rfloor$. Partition $V_0$ into subsets 
{$\{v_1, v_2\}, \ldots, \{v_{\nu-2}, v_{\nu-1}\}, \{v_\nu\}$} such that 
{$\{v_1v_2, \ldots, v_{v-2}v_{\nu-1}\}$} forms a maximum matching of $G$. Choose
$E_1=E_1'$ and $E_2=E_2'$. So, we have
$$\varepsilon_1+\varepsilon_2=e(H_1)=\frac{\nu(n-\nu)\Delta_2-1}2.$$
Therefore,
$$e(H)=\left\lfloor\frac{\nu(n-\nu)\Delta_2}{2}\right\rfloor+g(\nu,\Delta_2,\lfloor{\nu}/{2}\rfloor)=f(n,\nu,\Delta_2).$$

\end{proof}

\section{Proof of Theorem~\ref{THM: MAIN}}

Our proof is by induction on the matching number, in the following lemma we give the proof  of the base case.
\begin{lem}\label{LEM: L8}
If $H$ is  a $3$-graph on $n$ vertices with $\Delta_2(H)\leq\Delta_2$ and $\nu(H)=1$, then
 $e(H)\le f(n,1,\Delta_2)$ for $n\geq 32$. 
\end{lem}

\begin{proof}

{If all edges of $H$ intersect a fixed vertex $x\in V(H)$, then, by Lemma~\ref{LEM: L7}, $e(H)\le f(n,1,\Delta_2)$, as desired.
If not, choose a vertex $x$ of $H$ with maximum $1$-degree and let $L(x)$ be the link graph of $x$. Without loss of generality, we may assume $\Delta_1(H)\ge \Delta_2(H)\ge 2$ (otherwise, the result is trivial). If $\nu(L(x))=1$, then $L(x)$ must either be a triangle or a star. If $L(x)$ is a triangle spanned by $\{y,z,w\}$ say, then
each edge in $H$ must contains at least two vertices of $\{y,z,w\}$. So we have
\begin{equation*}
e(H)\le \binom{3}{2}\Delta_2\leq f(n,1,\Delta_2)
\end{equation*}
when $n\geq8$. If $L(x)$ is a star with centre $w$ say, we can assume that $d_1(x)=d_1(w)=\Delta_1(H)\ge 2$. Choose two edges $e=\{x,w,y\}$ and  $f=\{x,w,z\}$ in $H$.
Since $d_1(x)=\Delta_1(H)$, each edge of $H$ not containing $x$ does not contain $w$, too.
Since $\nu(H)=1$, each edge not containing $x$ must contain $y$ and $z$. 
So we have
\begin{equation*}
e(H)\le 2\Delta_2\leq f(n,1,\Delta_2)
\end{equation*}
when $n\ge 6$.
If $\nu(L(x))\geq 2$, let $\{w,y\}$ and $\{u,z\}$ be two independent edges in $L(x)$, then each edge not containing $x$ must intersect both of $\{w,y\}$ and $\{u,z\}$ as $\nu(H)=1$.  
Without loss of generality, assume $\{y,z,v\}\in E(H)$ and $v\not=x$, let $T=\{x,w,y,z,u,v\}$. Since $\nu(H)=1$,  every edge of $H$ must intersect $T$ at least $2$ vertices. We have
\begin{equation*}
e(H)\le \binom{6}{2}\Delta_2\leq f(n,1,\Delta_2)
\end{equation*}
when $n\ge 32$.
This completes the proof of the lemma.}
\end{proof}

Now we give  the proof of the main theorem.

\begin{proof}[Proof of Theorem~\ref{THM: MAIN}]
We proceed by induction on $\nu$. Lemma~\ref{LEM: L8} gives the proof of the base case. Suppose that the result holds for all 3-graphs with maximum codegree at most $\Delta_2$ and matching number at most $\nu-1$. Now let $H$ be a $3$-graph on $n$ vertices with $\Delta_2(H)\le\Delta_2$ and $\nu(H)=\nu$. Suppose to the contrary that $e(H)>f(n,\nu,\Delta_2)$. For any edge $e\in E(H)$, define $N(e)=\{f : f\cap e\not=\emptyset, f\in E(H)\}$. We claim that
\begin{eqnarray*}
|N(e)|\geq\frac{\Delta_2(n-\nu)}{4}.
\end{eqnarray*}
Otherwise, $e(H-V(e))> f(n,\nu,\Delta_2)-\frac{\Delta_2(n-\nu)}{4}>f(n-3,\nu-1,\Delta_2)$. By induction hypothesis, $\nu(H-V(e))\ge \nu$. This implies that $\nu(H)\ge \nu+1$,
a contradiction.

Since $e(H)>f(n,\nu,\Delta_2)\ge f(n,\nu-1,\Delta_2)$, by induction hypothesis,  we can find $\nu$ independent edges in $H$ and each of them contains a vertex of $1$-degree at least $\frac{\Delta_2(n-\nu)}{12}$ as $|N(e)|\geq\frac{\Delta_2(n-\nu)}{4}$ for any $e\in E(H)$.
Take a vertex of 1-degree at least  {$\frac{\Delta_2(n-\nu)}{12}$} from each of these $\nu$ independent edges and put them together as a set $U$. Denote $U=\{x_1, x_2,\ldots,x_\nu\}$.

\begin{claim}
$e(H-U)=0$.
\end{claim}
Suppose to the contrary that there is an edge $e_0\in E(H-U)$. We first show that we can greedily find $\nu+1$ independent edges in $H$. Assume we have found $t$ independent edges $e_0, e_1, \ldots, e_{t-1}$ for some $1\le t\le \nu$ such that $e_i\cap U=\{x_i\}$ for each $1\le i\le t-1$.  For $x_{t}$, note that $\Delta_2(H)\le \Delta_2$, we have
{
$$|\{e : x_{t}\in e \mbox{ and } e\cap(\cup_{i=0}^{t-1}e_i\cup (U\setminus\{x_{t})\})\not=\emptyset\}|\le (2t+\nu)\Delta_2\le  3\nu\Delta_2<\frac{\Delta_2(n-\nu)}{12}$$
when $n>37\nu$.}
Since $d_1(x_t)\ge \frac{\Delta_2(n-\nu)}{12}$, there is an edge $e_t$ such that $e_t\cap U=\{x_t\}$ and $e_0,e_1,\ldots,e_t$ are independent in $H$. The end of the process gives us a matching of $H$ with $\nu+1$ edges, a contradiction to $\nu(H)=\nu$.

By the above claim, all the edges of $H$ intersect a given subset of vertices of size $\nu$, Lemma~\ref{LEM: L7} implies that $e(H)\le f(n,\nu,\Delta_2)$, a contradiction to our assumption.
The proof is completed.
\end{proof}

\section{Concluding Remark}

In this paper, $3$-graphs are not necessarily simple, it is natural to ask what is the size of simple $3$-graphs under the same {restrictions}.
\begin{problem}\label{PROB: p1}
What is the maximum size of simple $3$-graphs on $n$ vertices with maximum codegree $\Delta_2$ and matching number $\nu$?
\end{problem}
It is plausible if the answer of Problem~\ref{PROB: p1} is $f(n,\nu,\Delta_2)$. To do this, the only difficulty is if we can construct simple MPTS$(\nu,\lambda)$ for $\lambda\leq \nu-2$. Dehon~\cite{MD-DM83} proved that there exists a simple TS$(\nu,\lambda)$ if and only if $\lambda\leq \nu-2$, $\lambda \nu(\nu-1)=0(\text{mod } 6)$ and $\lambda (\nu-1)=0(\text{mod } 2)$.
\begin{problem}\label{PROB: p2}
Do we have  sufficient and necessary condition for the existence of  a simple PTS$(\nu,\lambda)$?
\end{problem}




\begin{thebibliography}{99}

\bibitem{BK-DM09}
N. Balachandran, N. Khare, Graphs with restricted valency and matching number, Discrete Math., 2009, 309(12): 4176-4180.

\bibitem{CR-DM13}
J. Chaffee, C. A. Rodger, Group divisible designs with two associate classes, and quadratic leaves of triple systems, Discrete Math. 313 (2013) 2104-2114.

\bibitem{CH-DM76}
V. Chv\'{a}tal, D. Hanson, Degrees and matchings, J. Combin. Theory Ser. B, 1976, 20(2): 128-138.

\bibitem{CD-HCD07}
C. J. Colbourn, J. H. Dinitz, Handbook of Combinatorial Designs, second edition, Chapman \& Hall/CRC, Boca Raton, FL, 2007.

\bibitem{CR-LEN91}
C. J. Colbourn, A. Rosa, Leaves, excesses and neighbourhoods in triple systems, Austral. J. Combin. 4 (1991) 143-178.

\bibitem{CR-GC86}
C. J. Colbourn, A. Rosa, Quadratic leaves of maximal partial triple systems, Graphs Combin. 2 (4) (1986) 317-337.

\bibitem{CR-JG87}
C. J. Colbourn, A. Rosa, Element neighbourhoods in twofold triple systems, J. Geom. 30 (1) (1987) 36-41.

\bibitem{MD-DM83}
M. Dehon, On the existence of $2$-designs $S_\lambda(2,3,v)$ without repeated blocks, Discrete Math., 1983, 43(2-3): 155-171.





\bibitem{PE-PMI61}
P. Erd\H{o}s, A problem on independent r-tuples, Ann. Univ. Sci. Budapest. 8(1965) 93-95.

\bibitem{PT-JCTA65}

{P. Erd\H{o}s, T. Gallai, On maximal paths and circuits of graphs, Acta Math. Acad. Sci. Hungar. 10(1959), 337-356.}



\bibitem{FRR-CPC12}
P. Frankl, V. R\"{o}dl, A. Ruci\'{n}ski, On the maximum number of edges in a triple system not containing a disjoint family of a given size, Combin. Probab. Comput. 21 (2012) 141-148.

\bibitem{PF-JCTA13}
P. Frankl, Improved bounds for Erd\H{o}s' Matching Conjecture. J. Combin. Theory Ser. A 120(2013) 1068-1072.

\bibitem{HLS-CPC12}
H. Huang, P. Loh, B. Sudakov, The size of a hypergraph and its matching number, Combin. Probab. Comput. 21 (2012)442-450.

\bibitem{KT-NZ00}
P. Katerinis, N. Tsikopoulos, Minimum degree and F-factors in graphs, New Zealand J. Math. 29 (2000) 33-40.

\bibitem{KN-DM14}
N. Khare, On the size of 3-uniform linear hypergraphs, Discrete Math., 2014, 334: 26-37.


\bibitem{K-CDM47}
T. P. Kirkman, On a problem in combinations, Cambridge and Dublin Math. J. 2 (1847) 191-204.




\bibitem{RS-JCTA87}
C. A. Rodger, S. J. Stubbs, Embedding partial triple systems, J. Combin. Theory Ser. A 44 (2) (1987) 241-252.



\bibitem{SS-DM13}
A. P. Street, D. J. Street, Combinatorics of experimental design, Oxford University Press, Inc., 1986.



\end{thebibliography}
\end{document}